\documentclass[reqno,12pt]{amsart}

\usepackage{epsf}
\usepackage{graphics}
\usepackage{graphicx}
\usepackage{amssymb}
\usepackage{amsmath}
\usepackage{xcolor}
\usepackage{microtype}
\usepackage[bookmarks=true,pdfencoding=auto,hidelinks]{hyperref}

\date{}

\theoremstyle{plain}
\newtheorem{theorem}{Theorem}

\theoremstyle{definition}

\theoremstyle{remark}
\newtheorem{remark}[theorem]{Remark}

\def\C{{\mathbb C}}
\def\N{{\mathbb N}}

\title{Minimal Cobordisms between thin and thick torus knots}

\author{S.~Baader, L.~Lewark, F.~Misev, P.~Tru\"ol}

\begin{document}
\begin{abstract} We determine the locally flat cobordism distance between torus knots with small and large braid index, up to high precision. Here small means $2$, $3$, $4$, or $6$. As an application, we derive a surprising fact about torus knots that appear as cross-sections of almost minimal cobordisms between two-stranded torus knots and the trivial knot.
\end{abstract}

\maketitle

\section{Introduction}

The topological 4-genus $g_4$ is a hard to compute knot invariant, despite being well-studied. Even the precise 4-genus of torus knots remains unknown. The best known lower bound for $g_4$ is given by half the maximum of the Levine--Tristram signature function, see e.g.~\cite{C}. In~\cite{BBL}, the question is raised whether this bound is sharp for torus knots, and evidence is provided for infinite families of torus knots with braid index 3, 4, and 6.
In this paper, we study a relative version of this question.
Given two knots $K$ and $J$ in $S^3$, we define the \emph{locally flat cobordism distance $d_{\chi}(K,J)$} to be the absolute value of the largest Euler characteristic among all topological, locally flat, connected cobordisms in $S^3 \times [0,1]$ between $K$ and $J$. When $J=O$ is the trivial knot, we obtain $d_{\chi}(K,O)=2g_4(K)$, twice the \emph{topological 4-genus} of~$K$. For coprime positive integers $p, q$, we denote the corresponding positive torus knot by $T(p,q)$, and its classical signature invariant \cite{Tro,Mur} by~$\sigma(T(p,q))$. We use the convention that positive torus knots have positive signatures, e.g.~$\sigma(T(2,3))=2$.

\begin{theorem}
\label{bound246}
There exist positive constants $a,b>0$ with the following property.
For all $k \in \{2,4,6\}$ and for all $m,N \in \N$ with $\gcd(k,N)=1$ and $N \geq \frac{m^2}{k}+\frac{2m^2}{k^3}$:
$$\left| d_{\chi}(T(m,m+1),T(k,N))-\sigma(T(k,N))+\sigma(T(m,m+1)) \right| \leq am+b.$$
\end{theorem}

The order of growth of the signature invariants is approximately 
\begin{align*}
	\sigma(T(m,m+1))&\approx \frac{m^2}{2} \qquad \text{up to an affine error in} \, m,\\
		\text{and}\qquad  \sigma(T(k,N))&\approx \frac{kN}{2}>\frac{m^2}{2} \qquad \text{up to a constant error.}
\end{align*}
This follows from \cite[Theorem~5.2]{GLM}. 
Therefore the bound on the error term~${am+b}$ in Theorem~\ref{bound246} is asymptotically negligible. In the special case $m=1$, the knot $T(1,2)=O$ is trivial, so we obtain $d_{\chi}(O,T(k,N))\approx \sigma(T(k,N))$, up to a constant error of at most $a+b$. This is compatible with the known fact that $2g_4(T(k,N))\approx \sigma(T(k,N))$, up to a constant error, for $k\in \{2,3,4,6\}$; see~\cite{BBL}. In fact, our proof of Theorem~\ref{bound246} is based on the latter result. This is the reason for the restriction on the value of~$k$.

For odd parameters~$k$, the situation is slightly different, in that the signature does not provide the best lower bound for the corresponding cobordism distance (see e.g.~\cite[Section~3]{BBL}). Nevertheless, we still obtain a sharp result for~${k=3}$ up to an affine error. We denote by $\sigma_{\zeta_3}(K)$ the Levine--Tristram signature invariant of~$K$ at $\zeta_3=e^\frac{2 \pi i}{3}$, as defined in~\cite{L,T}, see also~\cite{C}.

\begin{theorem}
\label{bound3}
There exist positive constants $a,b>0$ with the following property.
For all $m,N \in \N$ with $\gcd(3,N)=1$ and $N \geq \frac{5}{12}m^2$:
$$\left| d_{\chi}(T(m,m+1),T(3,N))-\sigma_{\zeta_3}(T(3,N))+\sigma_{\zeta_3}(T(m,m+1)) \right| \leq am+b.$$
\end{theorem}

Here the order of growth of the signature invariants is
\begin{align*}
	{\sigma_{\zeta_3}(T(m,m+1))}&\approx\frac{4m^2}{9} \qquad \text{up to an affine error in} \, m,\\
		\text{and}\qquad  \sigma_{\zeta_3}(T(3,N))&\approx \frac{4m^2}{9} \qquad \text{up to a constant error,}
\end{align*}
see \cite[Proposition~5.2]{GG}. 
Therefore, as in Theorem~\ref{bound246}, the error term $am+b$ is asymptotically negligible.

We expect both Theorem~\ref{bound246} and ~\ref{bound3} to generalize to all even (respectively odd) values of~$k$, with the constants $a$, $b$ depending on~$k$, and appropriate lower bounds of the form $N \geq c(k)m^2$. In the case of odd~$k$, one would have to replace the signature invariant $\sigma_{\zeta_3}(K)$ by $\sigma_\omega(K)$ with~$\omega=e^{\pi i(k-1)/k}$.
However, in order to prove these generalizations, one would first have to establish the conjectural approximate equality $2g_4(T(k,N))\approx \sigma(T(k,N))$ for fixed $k$ and all $N$ up to a constant error.


\begin{remark}
We conclude the introduction with an interesting hypothetical consequence of the case $k=2$ in Theorem~\ref{bound246}, provided that the equality $$2g_4(T(m,m+1))\approx \sigma(T(m,m+1))$$
were true up to an affine error in~$m$.

By Theorem~\ref{bound246} for $k = 2$, up to an affine error in $m$, we have
$$d_{\chi}(T(m,m+1),T(2,N))\approx N-\frac{m^2}{2},$$
provided $N \geq \frac{3}{4}m^2$, since $\sigma(T(2,N))=N-1$ and $\sigma(T(m,m+1))\approx \frac{m^2}{2}$.
The hypothetical approximation $2g_4(T(m,m+1))\approx \sigma(T(m,m+1))$ would then imply
$$2g_4(T(2,N))=N-1\approx d_{\chi}(T(m,m+1),T(2,N))+2g_4(T(m,m+1)),$$
again up to an affine error in~$m$. 
This means that between the knot $T(2,N)$ and the trivial knot there would exist a locally flat cobordism of 
 minimal genus, up to an affine error in~$m$, with cross section the knot~$T(m,m+1)$.
A similar statement is unthinkable in the smooth setting, since every knot $K$
that appears as a cross-section
of a minimal genus smooth cobordism between the knot $T(2,N)$ and the trivial knot must satisfy $\sigma(K) = s(K)$ for the Rasmussen invariant $s$~\cite{R}. However, $\sigma(K) \approx s(K)/2$ for $K = T(m,m+1)$ up to an affine error in~$m$.
\end{remark}

The proofs of Theorems~\ref{bound246} and~\ref{bound3} are based on McCoy's twisting method~\cite{MC}, and are conceptionally different from the proofs of similar statements on smooth cobordisms, such as those in~\cite{B,BF,F}. Rather, our method is inspired by recent results on the topological 4-genus, in particular by~\cite{BBL} and~\cite{FMP}.

\section{Twisting up torus knots}

In the proofs of Theorems~\ref{bound246} and~\ref{bound3} we use McCoy's twisting method.
A \emph{null-homologous twist} is an operation that inserts a full twist into an even number of parallel strands of a knot, where the twisting circle and the knot are required to have linking number zero. Null-homologous twists come in two versions, \emph{positive} or \emph{negative}, depending on 
the sign of the inserted full twist. We will need one particular instance of a \emph{positive} null-homologous twist, as shown in Figure~1, where the twisting circle is colored blue. 
In this figure, all strands are assumed 
to have the same orientation, e.g.~from left to right, and the $-1$ and $+1$ in the two boxes refer to a negative and a positive full twist, respectively. Here the total number of strands (which is six in Figure~1) can be replaced by any even number $2n$.

\begin{figure}[htb]
\begin{center}
\raisebox{-0mm}{\includegraphics[scale=1.5]{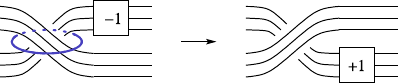}}
\caption{A positive null-homologous twist.}
\end{center}
\end{figure}

McCoy's main result states that if a knot $K$ is related to the trivial knot by a sequence of (at most) $x$ positive and $x$ negative null-homologous twists, then
$g_4(K) \leq x$ (see \cite[Theorem~1]{MC}).
In~\cite{BBL}, McCoy's method was used to determine the topological $4$-genus of all torus knots with braid index $k\in \{2,3,4,6\}$, up to a constant error. In fact, it was shown in~\cite{BBL} that all of these knots~$K$ can be transformed into the trivial knot using $\frac{1}{2}\sigma(K)$ positive null-homologous twists, up to a small constant error. For example, for $k=3$ and $K=T(3,N)$, the error is at most $2$.
We will use these facts in the proof of Theorem~\ref{bound246} below.

The other main idea of the proof is to relate the two main players of Theorem~\ref{bound246}, the torus knots $T(m,m+1)$ and $T(k,N)$, to two other knots, a knot denoted $K_k$ and the torus knot $T(k,n)$, for an appropriate choice of $n$. The latter two knots will fulfill $d_\chi(K_k,T(k,n))\approx \sigma(T(k,n))-\sigma(K_k)$, up to an affine error in $m$, which will imply the desired approximation for the former two knots. 
The knot $K_k$ is defined as the $(k,1)$-cable of the smaller torus knot $T(l,l+1)$, where $l=\lfloor m/k\rfloor$. The crucial insight behind this choice is that the knot $K_k$ is, in fact, close to the torus knot $T(m,m+1)$, in terms of $d_\chi$ and $\sigma$, up to an affine error in~$m$. A~sequence of unknotting crossing changes for $T(l,l+1)$ translates to an untwisting sequence for~$K_k$, which in turn will provide the desired estimates for $T(m,m+1)$. The following proof is based on these ideas.
%
\begin{proof}[Proof of Theorem~\ref{bound246}]
We prove the desired statement for all values ${k \in \{2,4,6\}}$ simultaneously. 
Throughout the proof, unless indicated otherwise, we write $\approx$ for equalities that hold up to an additive error term bounded by an affine function $E(m)=am+b$. We use the notation~$\sigma_i$ for the Artin braid generators.

Our first observation is that we need only consider the case where $m$ is a multiple of $k$. 
In fact, by choosing $l$ such that $m-k < kl \leq m$,
the standard braid diagram $(\sigma_1 \ldots \sigma_{m-1})^{m+1}$ of the torus knot $T(m,m+1)$ can be transformed into the standard braid diagram of the torus knot $T(kl,kl+1)$ by a sequence of at most $2(k-1)m$ saddle moves. Indeed, we can use $(m-kl)(m-1)+(m-kl)(kl+1) = (m-kl)(m+kl) \leq 2(k-1)m$ 
saddle moves to remove as many braid crossings; see e.g.~\cite[Figure 4]{BLMT}.
Thereby the values of all quantities involved in the statement, i.e.~the values of~$\sigma,d_\chi,g_4$, change only by an amount bounded by a fixed affine function $E(m)=am+b$. 
Thus, from now on we will work with the assumption~$m=kl$. 

Our next observation is that the standard braid diagram $(\sigma_1 \ldots \sigma_{l-1})^l$ of the torus link $T(l,l)$ can be transformed into a trivial $l$-component braid by a sequence of~$\frac{1}{2}(l-1)l$ positive-to-negative crossing changes. Indeed, after changing the last $l-1$ crossings of that braid diagram into negative crossings, the last strand of the braid ends up undercrossing all the other strands and can therefore be pulled straight using a braid isotopy. As a result, the first $l-1$ remaining strands form the standard braid for $T(l-1,l-1)$. Proceeding inductively, we obtain the trivial braid on $l$ strands after a total of $(l-1)+(l-2)+\ldots+1=\frac{1}{2}(l-1)l$ positive-to-negative crossing changes and corresponding braid isotopies. A more general argument is given in~\cite[Theorem~1.2]{KLMMS}.

As a consequence, the standard braid diagram $(\sigma_1 \ldots \sigma_{l-1})^{l+1}$ of the torus knot $T(l,l+1)$ can be transformed into the braid $(\sigma_1 \ldots \sigma_{l-1})$ representing the trivial knot by a sequence of $\frac{1}{2}(l-1)l$ positive-to-negative crossing changes. 

Let $C_{p,q}(K)$ denote the $(p,q)$-cable of a knot $K$. 
We continue by considering the link~$L_k = {C_{k,0}(T(l,l+1))}$ with $k$ components.
A natural diagram of $L_k$ is obtained from the braid diagram of $T(l,l+1)$ by replacing each crossing with the local picture shown on the left side of Figure~1 (for $k=3$). Here the negative full twist accounts for the correct framing. 
By the previous paragraph, using $\frac{1}{2}(l-1)l$ positive null-homologous twists, the link $L_k$ can be transformed into the $(k,0)$-cable of the trivial knot $T(l,1)$, in other words, the $k$-component trivial link.

The following little alteration transforms the link $L_k$ into a knot~$K_k$: insert the braid $(\sigma_1 \sigma_2 \cdots \sigma_{k-1})$ on the left of the braid diagram representing $L_k$. The construction is chosen so that the knot $K_k$ can be transformed into the trivial knot by $\frac{1}{2}(l-1)l$ positive null-homologous twists. Note that 
$K_k = C_{k,1}(T(l,l+1))$.
Using the fact that $k$ is even, the famous cabling formula for the signature invariant by Litherland~\cite[Theorem 2]{Lith}
implies $\sigma(K_k)=0$ (see also~\cite[Proposition~1]{C}). 

In summary, we have a knot $K_k$ of signature zero, related to the trivial knot by a sequence of $\frac{1}{2}(l-1)l$ positive null-homologous twists. 

Now we choose a torus knot $T(k,n)$ with $n \approx \frac{2m^2}{k^3}=\frac{2l^2}{k}$. More precisely, we choose $n \in \N$ with $\gcd(k,n)=1$ and $|n-\frac{2l^2}{k}|<k$, which is possible. 
The recursive formula for the signature invariant of torus links by Gordon, Litherland and Murasugi implies that $\sigma(T(k,n))\approx\frac{kn}{2}\approx l^2$, 
up to a constant error~\cite[Theorem~5.2]{GLM}.
By~\cite{BBL}, the torus knot $T(k,n)$ is related to the trivial knot by a sequence of $\frac{1}{2}\sigma(T(k,n))\approx \frac{1}{2}l^2$ positive null-homologous twists, up to a constant error. In turn, the connected sum of knots $K_k \# T(k,-n)$ is related to the trivial knot by a sequence of $\frac{1}{2}(l-1)l$ positive and $\frac{1}{2}l^2$ negative null-homologous twists. McCoy's result thus implies
$$g_4(K_k \# T(k,-n)) \leq \frac{1}{2}l^2,$$
up to a constant error.
On the other hand, since the signature is additive under
connected sum, we have 
$$\sigma(K_k \# T(k,-n))=\sigma(K_k)+\sigma(T(k,-n))\approx -\frac{kn}{2}\approx -l^2.$$
Thus $g_4(K_k \# T(k,-n))\approx \frac{1}{2}l^2$, 
since $|\sigma(K)|$ is a lower bound for~$2g_4(K)$. Using $2g_4(K_k \# T(k,-n)) = d_\chi(K_k,T(k,n))$, we obtain 
$$d_\chi(K_k,T(k,n))\approx l^2\approx \sigma(T(k,n))-\sigma(K_k).$$

In the final step, we simply increase the framing on both knots $K_k$ and $T(k,n)$ by the same amount, by inserting the braid
$$(\sigma_1 \sigma_2 \cdots \sigma_{k-1})^{kl^2}$$
into $k$ parallel strands on both knots.
This turns $K_k = C_{k,1}(T(l,l+1))$ into $C_{k,1+kl^2}(T(l,l+1))$,
which is a smooth cobordism consisting of less than $mk$ saddles
away from the link $C_{k,kl+kl^2}(T(l,l+1)) = T(kl,kl+k) = T(m,m+k)$. (To see the second-to-last equality, note that more generally, $C_{p,pqr}(T(q,r))=T(pq,pr)$ for positive integers $p,q,r$.)
It turns $T(k,n)$ into ${T(k,kl^2+n)}$,
which is a smooth cobordism consisting of less than $k^2$ saddles
away from $T(k,\frac{m^2}{k}+\frac{2m^2}{k^3})$.
We observe that inserting the braid $(\sigma_1 \sigma_2 \cdots \sigma_{k-1})^{kl^2}$ does not affect the cobordism distance of the knots $K_k$ and $T(k,n)$, nor the difference of their signature invariants, by more than a constant term. Similarly, the applied saddles do not affect $d_\chi$ and $\sigma$ by more than an affine error in $m$. 
Therefore we obtain
$$d_\chi(T(m,m+k),T(k,\frac{m^2}{k}+\frac{2m^2}{k^3}))\approx \sigma(T(k,\frac{m^2}{k}+\frac{2m^2}{k^3}))-\sigma(T(m,m+k)).$$
This proves Theorem~\ref{bound246} for $N=\frac{m^2}{k}+\frac{2m^2}{k^3}$. Here the change from $T(m,m+1)$ to $T(m,m+k)$ is not an issue, since we can transform $T(m,m+k)$ to $T(m,m+1)$ by $(k-1)(m-1)\leq 5(m-1)$ saddles. 

For larger parameters, we simply twist the knot $T(k,N)$ further up, using the fact that
$$d_\chi(T(k,N+x)),T(k,N))\approx \sigma(T(k,N+x))-\sigma(T(k,N)),$$
up to a constant error,
which is true for $k\in \{2,3,4,6\}$ thanks to~\cite{BBL}.

For each $k\in \{2,4,6\}$, the above proof provides us with constants $a, b > 0$ as in the statement of the theorem. To obtain global constants for all $k\in \{2,4,6\}$, we choose the maxima of these constants. 
\end{proof}

\begin{proof}[Proof of Theorem~\ref{bound3}]
The cabling construction used in the above proof for the cases $k\in \{2,4,6\}$ carries over to the case $k=3$. However, there is an important change in the computation of the signature invariant of $(3,1)$-cables: 
unlike in the case of $C_{k,1}(K)$ for even $k$, the cable knot~$C_{k,1}(K)$ for odd~$k$ does not generally have a vanishing signature for any knot $K$.

However, there is another value of the Levine--Tristram signature $\sigma_{\omega}$, for which this is true: $\omega=\zeta_{k}$, where $\zeta_{k} \in \C$ is any root of unity of order~$k$, 
see~\cite[Theorem 2]{Lith}. 
Therefore, in the case $k=3$, we can simply replace the classical signature invariant $\sigma(K)$ by $\sigma_{\zeta_3}(K)$. The attentive reader might object that in order to copy the proof of Theorem~\ref{bound246}, we also need the existence of a torus knot $T(3,n)$ with $\sigma(T(3,n))=l^2$, for all $l \in \N$, and the additional property that $T(3,n)$ can be unknotted by about $\frac{1}{2}l^2$ positive null-homologous twists. Surprisingly, this also works with $\sigma$ replaced by $\sigma_{\zeta_3}$, for the simple reason that
$$\sigma_{\zeta_3}(T(3,n))=\sigma(T(3,n)),$$
up to an error of at most two. This last equality, which is due to Gambaudo--Ghys \cite[Proposition~5.2]{GG}, can also be used to find the exact condition $N \geq \frac{5}{12}m^2$, as stated in Theorem~\ref{bound3}. The details are not important enough to be included here, since the real interest of Theorem~\ref{bound3} lies in the qualitative statement.
\end{proof}

\bigskip
\noindent
\texttt{sebastian.baader@math.unibe.ch}

\medskip
\noindent
\texttt{lukas.lewark@math.ethz.ch}

\medskip
\noindent
\texttt{filip.misev@mathematik.uni-regensburg.de}

\medskip
\noindent
\texttt{paulagtruoel@gmail.com}

\bigskip
\noindent
Mathematisches Institut, Sidlerstrasse 5, 3012 Bern, Switzerland

\medskip
\noindent
ETH Z\"urich, R\"amistrasse 101, 8092 Z\"urich, Switzerland

\medskip
\noindent
Universit\"at Regensburg, 93053 Regensburg, Germany

\medskip
\noindent
Max-Planck-Institut f\"ur Mathematik, Vivatsgasse 7, 53111 Bonn, Germany

\end{document}